\documentclass[12pt]{amsart}
 \usepackage{mathtools}
\newtheorem{innercustomgeneric}{\customgenericname}
\providecommand{\customgenericname}{}
\newcommand{\newcustomtheorem}[2]{%
  \newenvironment{#1}[1]
  {%
   \renewcommand\customgenericname{#2}%
   \renewcommand\theinnercustomgeneric{##1}%
   \innercustomgeneric
  }
  {\endinnercustomgeneric}
}
\usepackage{amsmath,bbm, color}
\usepackage{latexsym,mathrsfs,bm}
\allowdisplaybreaks
\newcustomtheorem{formalresult}{Formal Result}
\usepackage{url}
\usepackage{paralist}
\usepackage{amsthm,amsfonts,amssymb,amsmath}
\usepackage{comment}

\usepackage[latin1]{inputenc}

\newcommand\N{\mathbb{N}}

\newcommand\R{\mathbb{R}}
\newcommand\Z{\mathbb{Z}}
\newcommand\C{\mathbb{C}}

\newtheorem{thm}{Theorem}[section]

\newtheorem{corollary}[thm]{Corollary}

\newtheorem{lemma}[thm]{Lemma}
\newtheorem{conj}[thm]{Conjecture}
\theoremstyle{remark}

\newtheoremstyle{named}{}{}{\itshape}{}{\bfseries}{.}{.5em}{\thmnote{#3}}
\theoremstyle{named}

\usepackage{color}

\let\originalleft\left
\let\originalright\right
\renewcommand{\left}{\mathopen{}\mathclose\bgroup\originalleft}
\renewcommand{\right}{\aftergroup\egroup\originalright}

 \allowdisplaybreaks

\numberwithin{equation}{section}

\begin{document}	

\newcommand{\powerlog}{Q}
\newcommand{\glog}{{h_Q}}

\newcommand{\powerpoly}{P}
\newcommand{\gpoly}{{g_P}}

\newcommand{\doesnotdivide}{\not\hspace{2.5pt}\mid}
\newcommand{\divides}{\mid}
\newcommand{\oldkappa}{s}
\newcommand{\oldkfirst}{r_1}
\newcommand{\oldksecond}{r_2}

\title{Shifted convolution sums motivated by string theory}
\author{Ksenia Fedosova \& Kim Klinger-Logan}
\begin{abstract}
In \cite{CGPWW2021}, it was conjectured that a particular shifted sum of  even divisor sums vanishes, and in \cite{SDK}, a formal argument was given for this vanishing. Shifted convolution sums of this form appear when computing the Fourier expansion of coefficients for the low energy scattering amplitudes in type IIB string theory \cite{GMV2015} and have applications to subconvexity bounds of $L$-functions. In this article, we generalize the argument from~\cite{SDK} and rigorously evaluate shifted convolution of the divisor functions of the form $\displaystyle
\sum_{\stackrel{n_1+n_2=n}{n_1, n_2 \in \mathbb{Z} \setminus \{0\}}} \sigma_{\oldkfirst}(n_1) \sigma_{\oldksecond}(n_2) |n_1|^\powerpoly $ and $\displaystyle \sum_{\stackrel{n_1+n_2=n}{n_1, n_2 \in \mathbb{Z} \setminus \{0\}  }} \sigma_{\oldkfirst}(n_1) \sigma_{\oldksecond}(n_2) |n_1|^\powerlog \log|n_1|
$ where $\sigma_\nu(n) = \sum_{d \divides n} d^\nu$. In doing so, we derive exact identities for these sums and conjecture that particular sums similar to but different from the one found in \cite{CGPWW2021} will also vanish. \\

{\small {\it keywords:} shifted convolution sum, divisor function, Hurwitz zeta function}
\end{abstract}
\maketitle

\section{Introduction}

Shifted convolution sums have a long history of being studied by number theorists \cite{Blomer2004,cho2013evaluation, hahn2007convolution, huard2002elementary, kim2013convolution, lemire2006evaluation, Michel2004, ntienjem2017evaluation,   park2022multinomial}. Recently, the AdS/CFT correspondence and Yang-Mills theory
gives a surprising hint towards the exact evaluation of shifted convolution sums of  divisor functions. Namely, \cite{CGPWW2021} conjectured that for any $n\neq 0$ 
and 
\begin{align*}
\varphi(n_1,n_2)=  30 -\tfrac{n^2}{4 n_1^2}-\tfrac{n^2}{4 n_2^2}-\tfrac{3 n}{n_1}-\tfrac{3 n}{n_2}+\left(15-\tfrac{30 n_1}{n}\right) \log |n_1|+\left(15-\tfrac{30 n_2}{n}\right) \log |n_2|,
\end{align*}
the following equality holds:
\begin{equation}\label{eq:Chester_conjecture}
\displaystyle\sum_{\stackrel{n_1+n_2=n}{n_1, n_2 \in \mathbb{Z} \setminus \{ 0\}}} \varphi(n_1,n_2) \sigma_2(n_1) \sigma_2(n_2)=
2 \sigma_2(n) \left( \tfrac{\zeta(2) n^2}{4} + 15 \zeta'(-2) \right),
\end{equation}
where  $\zeta$ denotes the Riemann zeta function.
The fact that this sum is both infinite and and involves {\it even} divisor functions makes it particularly challenging to work with.
In \cite{SDK}, the authors give a formal argument verifying \eqref{eq:Chester_conjecture}. 

In this paper, we examine what can be rigorously proven using the argument presented in \cite{SDK}. We apply similar methods to extend these results to more general convolution sums and give a precise equality. 
More explicitly, we 
{\it rigorously} evaluate for any $n \in \mathbb{Z} \setminus \{0\}$ and for certain $\oldkfirst,\oldksecond, \powerpoly, \powerlog  \in \mathbb{C}$,
\begin{equation}\label{eq:sums}
\sum_{\stackrel{n_1+n_2=n}{n_1, n_2 \in \mathbb{Z}  \setminus \{0\} }} \sigma_{\oldkfirst}(n_1) \sigma_{\oldksecond}(n_2) |n_1|^\powerpoly \text{ and } \sum_{\stackrel{n_1+n_2=n}{n_1, n_2 \in \mathbb{Z}  \setminus\{0\} }} \sigma_{\oldkfirst}(n_1) \sigma_{\oldksecond}(n_2) |n_1|^\powerlog \log|n_1|\end{equation}
in Theorems  \ref{thm:final_exact_result} and \ref{thm:final_exact_result2} respectively. In doing so we identify the obstacle to making the formal argument found in \cite{SDK} rigorous. The exact results found contain extra terms which do not appear in \eqref{eq:Chester_conjecture}. If we further assume that we may interchange the operations of taking an infinite sum and meromorphic continuation, this generalizes the argument in \cite{SDK}, and we recover \eqref{eq:Chester_conjecture} as desired.
With the help of this non-rigorous argument, we obtain other identities of a similar form, see Conjecture \ref{conj:zeros}.

The initial motivation for the study of sums  \eqref{eq:sums} is their appearance in string theory. The sum \eqref{eq:Chester_conjecture} arises in the Fourier modes of the homogeneous solution to differential equations involving Eisenstein series which yield coefficients for to the low energy scattering amplitude in type IIB string theory \cite{GMV2015}. More generally, sums of a similar form arise in the maximally supersymmetric $\mathcal{N}=4$ super-Yang-Mills theory when studying duality properties of certain correlation functions in the $1/N$ expansion \cite{FKL}.
However, as previously noted, shifted convolution sums more generally are of great interests to number theorists as well. Specifically, certain information on shifted convolution sums could yield progress on subconvexity problems for modular $L$-functions \cite{Blomer2004, Michel2004}.

\subsection{Main results}
 Let $c,d\in \mathbb{N}$. If $\gcd(c,d)$  divides~$n$, we let 
 \begin{align}\label{eq:definition_B}
     B := \{ b \in \mathbb{Z}: \exists\, a \in \mathbb{Z} \text{ such that } a d - b c = n \},\\
\label{def:bstar}
 b^* := \arg \min_{b \in B} \{ |b|\}\quad \text{and} \quad a^* := \frac{n+b^* c}{d}. 
 \end{align}
It is convenient to introduce
\begin{align}
u_{c,d} &:= \frac{ b^* \gcd(c,d)}{d}. \label{def:xn} 
\end{align}
Moreover, we denote 
$\displaystyle
\delta_{2 \divides P} = \begin{cases}
   1, \quad & \text{for }P \in 2 \Z, \\ 
   0, \quad & \text{otherwise.}
\end{cases}
$
 \begin{thm}\label{thm:final_exact_result}
		Let $n \in \Z \setminus \{0\}$, $\oldksecond\in \R$, $r_1 \in \C$ with $\text{Re}(\oldkfirst)<-1$, and $\powerpoly \in \mathbb{Z}$ be such that $\powerpoly \ll_{r_2} 0$ and let  $u_{c,d}$ be as in \eqref{def:xn}. Then 
		\begin{align}
		\sum_{ \stackrel{ n_1+n_2 = n}{ n_1, n_2 \in \mathbb{Z} \setminus \{ 0\}} } & \sigma_{\oldkfirst}(n_1) \sigma_{\oldksecond}(n_2) n_2^\powerpoly  = - n^\powerpoly \sigma_{\oldksecond}(n) \zeta(-\oldkfirst) \nonumber \\
		&+ (-1)^ \powerpoly 2  \delta_{2 \divides \powerpoly} \zeta(-\powerpoly) \zeta(-\oldksecond-\powerpoly) \sum_{d \divides n} d^{\oldkfirst+\powerpoly} \prod_{p | d} \left( 1 + (1-p^{\powerpoly})(p^{\oldksecond}+\ldots + p^{v_{p}(d) \oldksecond }) \right)  \label{eq:target_sum1}\\
		& + (-1)^\powerpoly \sum_{\stackrel{d \in \mathbb{N}}{d \doesnotdivide n}}   d^{\oldkfirst+\oldksecond+2 \powerpoly}   \sum_{\stackrel{0<c'\le d}{\gcd(c',d)|n}}  \frac{ \zeta(-\oldksecond-\powerpoly,\frac{c'}{d})}{\gcd(c',d)^{\powerpoly}}  \cdot \sum_{m\in \mathbb{Z}} (m+u_{c',d})^{\powerpoly}, \nonumber
		\end{align}		
  where the product $\prod_{p | d}$ is taken over all prime divisors $p$ of $d$ while the sum $\sum_{d | n}$ is taken over all divisors $d$ of $n$, and $v_p(d)$ denotes the valuation of $d$ at $p$.
		The first two terms in the right hand side of \eqref{eq:target_sum1} admit meromorphic continuations\footnote{Note that $2 \delta_{2 | P}$ may be continued as a meromorphic function in many ways. Specifically, we choose $e^{\pi i P} + 1 $ for $P \in \C$.} to $\powerpoly \in \mathbb{C}$ and $\oldkfirst \in\mathbb{C}$. The restriction of the second term to  $\oldksecond \in 2 \mathbb{Z}$ and $\oldksecond + \powerpoly \in \mathbb{N}_0$ equals 
		\[ \begin{cases}
		 - \zeta(\oldksecond) \sigma_{\oldkfirst}(n), \quad & \oldksecond +  \powerpoly = 0, \\
		 0, \quad & \oldksecond +  \powerpoly \in \N.
		\end{cases}
		\]
		Moreover, under the same condition, an inner sum in the third term admits a meromorphic continuation to $\powerpoly\in\mathbb{C}$ and vanishes:
$\displaystyle
\sum_{\stackrel{0<c'\le d}{\gcd(c',d)|n}}  \frac{ \zeta(-\oldksecond-\powerpoly,\frac{c'}{d})}{\gcd(c',d)^{\powerpoly}}  \cdot \sum_{m\in \mathbb{Z}} (m+u_{c',d})^{\powerpoly} = 0.
$
	\end{thm}
It is very tempting to change the order of taking a meromorphic continuation and an infinite sum and thus deduce from Theorem \ref{thm:final_exact_result} that for $n \in \mathbb{Z} \setminus \{ 0 \}$ and   $\oldkfirst\in\mathbb{R}-\{-1\}$,
	\begin{align}\label{eq:false_equation1a}&
	\sum_{ \stackrel{ n_1+n_2 = n}{ n_1, n_2 \in \mathbb{Z} \setminus \{ 0\}} } \sigma_{\oldkfirst}(n_1) \sigma_{\oldksecond}(n_2) n_2^\powerpoly 
	\end{align}
	appears to be equal to 
		\begin{align}\label{eq:false_equation1}- n^ \powerpoly \sigma_{\oldksecond}(n) \zeta(-\oldkfirst)+ \begin{cases}
	- \zeta(\oldksecond) \sigma_{\oldkfirst}(n), &\quad \oldksecond +\powerpoly = 0, \\
	0, & \quad \oldksecond + \powerpoly \in \N.
	\end{cases}
	\end{align}
However, we cannot justify the vanishing of the last term in \eqref{eq:target_sum1} and thus cannot infer that \eqref{eq:false_equation1a} is equal to \eqref{eq:false_equation1}.

The following statement is a result analogous to Theorem~\ref{thm:final_exact_result} but where each summand is multiplied by $\log|n_2|$. 
	\begin{thm}\label{thm:final_exact_result2}
		Let $n \in \Z \setminus \{0\}$,  $\text{Re}(\oldkfirst)<-1$, $\powerlog \in \mathbb{Z}$ be such that $\powerlog \ll_{r_2} 0$, and ${\text{Re}(\oldksecond)<-2-Q}$ and let   $u_{c,d}$ be as in \eqref{def:xn}. Then 
		\begin{align}\label{eq:target_sum2}
		&  \sum_{ \stackrel{ n_1+n_2 = n}{ n_1, n_2 \in \mathbb{Z} \setminus \{ 0\}} } \sigma_{\oldkfirst}(n_1) \sigma_{\oldksecond}(n_2) n_2^\powerlog \log|n_2| 
		 =  - \zeta(-\oldkfirst) \sigma_{\oldksecond}(n) n^\powerlog \log |n|\nonumber \\
		& \ \ \ \ \ \ \ \ \ \ \ \ + 2 (-1)^\powerlog   \sum_{d \divides n} \delta_{2 \divides \powerlog}\, d^{\,\powerlog+\oldkfirst} \zeta(-\powerlog) \left( \frac{\partial}{\partial t_1} + \frac{\partial}{\partial t_2} \right) \left. \left( \sum_{c \in \N} c^{t_1} \gcd(c,d)^{t_2} \right)\right|_{\substack{t_1=\oldksecond+\powerlog \\t_2=-\powerlog}}\nonumber \\
		& \ \ \ \ \ \ \ \ \ \  \ \ + 2 (-1)^\powerlog  \sum_{d \divides n} \delta_{2 \divides \powerlog}\, d^{\,\powerlog+\oldkfirst} (\log(d) \zeta(-\powerlog) -\zeta'(-\powerlog))  \sum_{c \in \N} \frac{c^{\oldksecond+\powerlog}}{\gcd(c,d)^\powerlog} \\
		& \ \ \ \ \ \ \ \ \ \  \ \ + (-1)^\powerlog \sum_{\stackrel{d \in \N}{d \doesnotdivide n}} d^{\oldkfirst} \left(\frac{ c d}{\gcd(c,d)}\right)^{\powerlog} \log \left| \frac{ c d}{\gcd(c,d)} \right| \sum_{m \in \mathbb{Z}} (m+u_{c,d})^{\powerlog} \nonumber \\
		&\ \ \ \ \ \ \ \ \ \  \ \  + (-1)^\powerlog \sum_{\stackrel{d \in \N}{d \doesnotdivide n}} \frac{d^{\oldkfirst+\oldksecond+ 2 \powerlog}}{(\gcd(c,d))^\powerlog} \sum_{\stackrel{0 < c' \le d}{\gcd(c', d) \divides n}} \zeta(-\oldksecond-\powerlog, \tfrac{c'}{d})  \sum_{m \in \mathbb{Z}}\frac{ \log|m + u_{c',d}|}{ (m+u_{c',d})^{-\powerlog}  } . 	\nonumber	\end{align}
		The first three terms admit a meromorphic continuation\footnote{As before, note that $2 \delta_{2 | Q}$ may be continued as a meromorphic function in many ways. Specifically, we choose $e^{\pi i Q} + 1 $ for $Q \in \C$.} in $Q$ and $\oldkfirst$  to $\C$.
		For certain values of $\oldkfirst, \oldksecond, \powerlog$ the first three terms significantly simplify; specifically, 
		\begin{align*}
		    \begin{cases}
		    - \zeta(-\oldkfirst) \sigma_{\oldksecond}(n) n^\powerlog \log |n|, & \quad  \powerlog \in \mathbb{N}$ and $\oldksecond + \powerlog \in \mathbb{N}, \\
		    \sigma_{\oldkfirst}(n) \zeta'(-\oldksecond), & \quad \powerlog =0\text{ and } \oldkfirst, \oldksecond \in \N, \\
		   (\frac{\log |n|}{2}-\log (2 \pi ))\sigma_0(n)
		   , & \quad \powerlog =0$ and $\oldkfirst= \oldksecond=0.
		    \end{cases}
		\end{align*}
		The meromorphic continuations of the inner sums of the last two terms to $\powerlog \in \mathbb{N}_0$ vanish; namely, \[		\sum_{m \in \mathbb{Z}} (m+u_{c,d})^{\powerlog} = 0 \]
		and 
		\[
		\sum_{\stackrel{0 < c' \le d}{\gcd(c'd, d) \divides n}} \zeta(-\oldksecond-\powerlog, \tfrac{c'}{d})  \sum_{m \in \mathbb{Z}} (m+u_{c',d})^\powerlog \log|m + u_{c',d}| = 0.
		\]
	\end{thm}	
	
	 Again, it is tempting to change the order of taking a meromorphic continuation and an infinite sum and thus deduce from Theorem \ref{thm:final_exact_result2} that for $\oldkfirst$ and $\oldksecond$ as above and $n \in \mathbb{Z} \setminus \{ 0 \}$, 
\begin{align}
\sum_{ \stackrel{ n_1+n_2 = n}{ n_1, n_2 \in \mathbb{Z} \setminus \{ 0\}} } \sigma_{\oldkfirst}(n_1)& \sigma_{\oldksecond}(n_2) n_2^\powerlog \log|n_2| 
\end{align} appears to be equal to
\begin{align}\label{eq:false_equation2} \begin{cases}
		    - \zeta(-\oldkfirst) \sigma_{\oldksecond}(n) n^\powerlog \log |n|, & \quad  \powerlog \in \mathbb{N}$ and $\oldksecond + \powerlog \in \mathbb{N}, \\
		    \sigma_{\oldkfirst}(n) \zeta'(-\oldksecond), & \quad \powerlog =0\text{ and } \oldkfirst, \oldksecond \in \N, \\
		   (\frac{\log |n|}{2}-\log (2 \pi ))\sigma_0(n)
		   , & \quad \powerlog =0$ and $\oldkfirst= \oldksecond=0.\\
		    \end{cases}
\end{align}
However, we cannot justify the vanishing of the last two terms of \eqref{eq:target_sum2} and so we cannot deduce this identity. 

At the same time, if we assume \eqref{eq:false_equation1} and \eqref{eq:false_equation2}, we would recover \eqref{eq:Chester_conjecture}. Moreover,
we obtain other conjectural identities which are supported by numerical evidence. 
\begin{conj}\label{conj:zeros}
	For any $n \in \mathbb{Z} \setminus \{ 0\}$, 
	\begin{equation}\label{eq:divisor_formula_1}
	\sum_{\stackrel{n_1, n_2 \in \mathbb{Z} \setminus \{ 0\}}{n_1+n_2=n}}\sigma_0(n_1)\sigma_0(n_2) \Big[ 2 + \tfrac{n_2-n_1}{n} \log |  \tfrac{n_1}{n_2}  |   \Big] = \sigma_0(n)  (2-\log \left(4 \pi ^2 |n|\right) ) .   
	\end{equation}
Moreover, for $n_1 n_2 \neq 0$, set \begin{align*}\varphi(n_1,n_2) = &  \tfrac{11}{3} -20 n_1 n^{-1}+20 n_1^2 n^{-2}  \\
 &  +\left(1 -12 n_1 n^{-1}+30 n_1^2 n^{-2}-20 n_1^3 n^{-3}\right) \log \left|n_1\right| \\
   & +\left(1 -12 n_2 n^{-1}+30 n_2^2 n^{-2}-20 n_2^3 n^{-3}\right) \log \left|n_2\right|.
 \end{align*}
 Then we have
\begin{align}
    \sum_{\stackrel{n_1,n_2 \in \mathbb{Z}\setminus \{ 0 \} }{n_1+n_2=n}}  \sigma_0(n_1) &  \sigma_0(n_2) \varphi(n_1,n_2)
   =\sigma_0(n) \left( \tfrac{11}{3}- \log (4 \pi^2 |n|)\right).
\end{align}
\end{conj}
In general, the numerical evidence does not support \eqref{eq:false_equation1} or \eqref{eq:false_equation2}.\footnote{ That is, take $d>0$ and let 
\begin{equation*}\label{eq:varphi}
\varphi(n_1,n_2) = \sum_{j=0}^{d}  \left( a_j n_1^j +  b_j n_2^j +  c_j n_1^j \log|n_1| + d_j n_2^j \log |n_2| \right),
\end{equation*}
with $a_j,b_j,c_j, d_j \in \C$ chosen in such a way that $\varphi(n_1,n-n_1) = o(|n_1|^{-1})$. If the informal argument could be justified in the current form, then the sum $\sum_{\stackrel{n_1,n_2 \in \mathbb{Z}\setminus \{ 0 \} }{n_1+n_2=n}}  \sigma_0(n_1)   \sigma_0(n_2) \varphi(n_1,n_2)$ should be equal to 
\[
\sigma_0(n) \left[ a_0 + b_0 + (c_0+d_0)\log (\sqrt{|n|}/2 \pi)+ \frac{1}{2} \sum_{j=1}^d \left(a_j + b_j +  (c_j+d_j) \log|n| \right) n^j \right].
\]
However, while the numerical evidence supports the claim for $d\le 4$, for $d \ge 5$ the equality doesn't hold, as calculations with the help of Pari/GP indicate. Similarly, an informal argument cannot be justified for $\oldkfirst \neq 0, \oldksecond \neq 0$.} 
As we will prove in an upcoming article with Danylo Radchenko, the differences between sums similar to those in Conjecture \ref{conj:zeros} and their informal evaluations with the help of \eqref{eq:false_equation1} and \eqref{eq:false_equation2} depend on Fourier coefficients of certain Hecke eigenforms. 

\subsection{Related research}
Sums involving divisor functions have received a considerable attention from number theorists due to their connection to the subconvexity of $L$-functions and other problems, see \cite{cho2013evaluation, hahn2007convolution, huard2002elementary, kim2013convolution, lemire2006evaluation, ntienjem2017evaluation,  park2022multinomial, Ramanujan}. However, much of what has been classically studied relates to odd divisor functions, that is, $\sigma_\nu(\cdot )$ for odd $\nu$.  Furthermore, the mentioned sources usually demand that $n_1$ and $n_2$, while satisfying $n_1+n_2=n$, belong to a finite set. For example, it is typical to examine truncated shifted convolution sums where $0 < n_1 < n$ and $0 < n_2 < n$. 
      
There are, however, results which do not demand that $n_1$ and $n_2$ belong to a finite set. In \cite{MR4163822}, Diamantis studies 
    \begin{equation}\label{eq:diamantis_convolution}
    \sum_{\stackrel{n_1 \in \mathbb{N}}{n_1>h}} \sigma_\alpha(n_1) \sigma_\beta(h-n_1) n_1^{-s}
    \end{equation}
    for  $h\in \mathbb{Z} \text{ and } \alpha, \beta, s \in \mathbb{C}$ was considered.
There, the author  characterizes the ratios of non-critical values of $L$-functions, corresponding to normalized weight $k$ cuspidal eigenforms,  in terms of  \eqref{eq:diamantis_convolution}. Additionally, Diamantis analytically continues \eqref{eq:diamantis_convolution}  in $s$ (although not to the whole $\mathbb{C}$)  by expressing it as 
a sum of Estermann $L$-functions (which in turn are linear combinations of Hurwitz zeta functions). We note that Theorem \ref{thm:final_exact_result} also expresses a convolution sum in terms of Hurwitz zeta functions; however,  it does not appear to be of the form as in \cite{MR4163822}.

\subsection{Acknowledgments} The authors would like to thank Michael Green for originally suggesting this problem. The authors would also like to thank Danylo Radchenko and Stephen D.\,Miller for their insightful conversations.  K.\,K-L. acknowledges support from NSF DMS-2001909.

\section{Preliminaries}
In this section, we include supporting lemmas needed for the main theorems. Specifically, Lemma~\ref{lemma:Hurwitz_equality} is a known result regarding Hurwitz zeta functions.  We also extend upon some results proven in \cite{SDK}  (Lemma \ref{lemma:Lemma5.1inKLSR} and Lemma \ref{lemma:rewriting_stuff_in_terms_of_Hurwitz_zeta}\footnote{ Lemma \ref{lemma:rewriting_stuff_in_terms_of_Hurwitz_zeta} is given in \cite{SDK} for a special case but is extended here without much change in the argument for this more general set up. }).

For $\oldkappa \in \mathbb{Z}$ with $\oldkappa$ sufficiently large and $a  \in \C $ with $\text{Re}(a) \in (0,1)$,
\begin{align}
\sum_{m \in \mathbb{Z}} (m+a)^{-\oldkappa}
&=\zeta (\oldkappa,a) +  e^{\pi i \oldkappa} \zeta (\oldkappa,1-a). \label{eq:regularization_of_infinite_sum}
\end{align}
We note 
that the right hand side of \eqref{eq:regularization_of_infinite_sum} is defined for all $a \in \C$ and is a meromorphic function in $\oldkappa \in \mathbb{C}$. 
Similarly, for $\oldkappa \in \mathbb{Z}$ with $\text{Re}(\oldkappa)$ sufficiently large and $a \in \C $ with $\text{Re}(a) \in (0,1)$,
\begin{align}
\sum_{m \in \mathbb{Z}}  (m+a)^{-\oldkappa} \log |m+a| 
&= \partial_\oldkappa \zeta (\oldkappa,a) +  e^{\pi i \oldkappa} \partial_\oldkappa \zeta (\oldkappa,1-a).\label{eq:regularization_of_infinite_sum2}
\end{align}

For the following lemma, we write 
$ \displaystyle\sum_{m \in \mathbb{Z}} (m+a)^{-\oldkappa} $ and  $ \displaystyle\sum_{m \in \mathbb{Z}} (m+a)^{-\oldkappa} \log | m+a|$  in the sense of meromorphic continuation as in 
~\eqref{eq:regularization_of_infinite_sum} and \eqref{eq:regularization_of_infinite_sum2}.

\begin{lemma}\label{lemma:Hurwitz_equality}
For $\oldkappa \in - \N_0$ and any $a \in \mathbb{C}$, 
	\begin{align}\label{ref:vanishing_for_bernoulli}
	\zeta(\oldkappa , 1-a) = (-1)^{\oldkappa +1} \zeta(\oldkappa , a)
	\end{align}	
	and for any $\oldkappa \in \mathbb{Z}$ and $a \in \C \setminus  \mathbb{Z}$,
	\begin{align}\label{eq:lemma2.2full_sum}
	    \sum_{m \in \mathbb{Z}} (m+a)^{-\oldkappa} = 	(-1)^{\oldkappa  }    \sum_{m \in \mathbb{Z}} (m-a)^{-\oldkappa}
	\end{align}
	and 
		\begin{align}\label{eq:lemma2.2full_sum_log}
	  \sum_{m \in \mathbb{Z}}  (m+a)^{-\oldkappa} \log |m+a| = 	(-1)^{\oldkappa  }    \sum_{m \in \mathbb{Z}}  (m-a)^{-\oldkappa} \log |m-a|.
	\end{align}
Moreover, for $a \in \C \setminus  \mathbb{Z}$ and $\oldkappa \in - \mathbb{N}_0$,
	\begin{align}\label{eq:crazy_vanihsing}
	    \sum_{m \in \mathbb{Z}} (m+a)^{-\oldkappa} = 0.
	\end{align}
\end{lemma}
 \begin{proof}
 Without loss of generality assume $\text{Re}(a) \in (0,1)$.  
	We note that values of $\zeta(\oldkappa,a)$ for $\oldkappa \in -\mathbb{N}$ are related to the Bernoulli polynomials~\cite[(25.11.14)]{NIST}:
	\[
	\zeta(\oldkappa, a) =  \frac{B_{-\oldkappa + 1}(a)}{\oldkappa	-1},
	\]
	and the Bernoulli polynomials satisfy  \cite[(24.4.3)]{NIST}:
	\[
	B_{-\oldkappa + 1}(a) = (-1)^{-\oldkappa + 1} B_{-\oldkappa + 1}(1-a).
	\]
This proves the first statement of the lemma. 
	The equality \eqref{eq:lemma2.2full_sum} follows from \eqref{eq:regularization_of_infinite_sum}:
	\begin{align} 
	(-1)^{-\oldkappa} \zeta (\oldkappa,1-a)+\zeta (\oldkappa,a)
	 =   (-1)^{-\oldkappa} \left[ \zeta (\oldkappa,1-a)+ (-1)^{-\oldkappa}  \zeta (\oldkappa,1-(1-a))  \right]. \nonumber
	\end{align}
	Similarly, \eqref{eq:lemma2.2full_sum_log} can be obtained similar to \eqref{eq:lemma2.2full_sum}. 
In order to show \eqref{eq:crazy_vanihsing}, we note that by 	\eqref{ref:vanishing_for_bernoulli}, \[ (-1)^{-\oldkappa} \zeta (\oldkappa,1-a)+\zeta (\oldkappa,a) = ((-1)^{-\oldkappa-\oldkappa+1} +1) \zeta (\oldkappa,a)\] vanishes.
\end{proof} 

The following lemma is contained in \cite[Lemma 5.1]{SDK}. However, for the convenience of the reader, we added more details to the proof of the statement.
\begin{lemma}\label{lemma:Lemma5.1inKLSR} For $\text{Re} (s)>1 $, $\text{Re} (s+k)>1 $  and $d \in \mathbb{N}$,
\begin{align}\label{dirichlet_series_with_loggcd}
\sum_{c \in \mathbb{N}} \frac{\log | \gcd(c,d) | }{c^s} = \zeta(s) \sum_{\ell | d} \Lambda(\ell) \ell^{-s},    
\end{align}
	where $\Lambda$ is the von Mangoldt function. Moreover, 
\begin{align}\label{dirichlet_series_with_gcd}
\sum_{c \in \mathbb{N}} \frac{  (\gcd(c,d))^k }{c^{s+k}} = \zeta(s+k) \prod_{p | d} \left( 1 + (1-p^{-k})(p^{-s}+\ldots + p^{-v_{p}(d)s }) \right),    
\end{align}
where the product is taken over all possible prime divisors $p$ and  $v_{p}(d)$ denotes the valuation of~$d$ at prime $p$. 
The sum in \eqref{dirichlet_series_with_gcd} admits a meromorphic continuation in both $s$ and $k$, for which the following holds:
\begin{align}\label{eq:evaluations_gcd_dirichlet}
    \zeta(s+k) \prod_{p | d} \left( 1 + (1-p^{-k})(p^{-s}+\ldots + p^{-v_{p}(d)s }) \right)  = \begin{cases}
0, \quad & s+k \in - 2 \mathbb{N}, \\ 
- d^k / 2, \quad & s+k = 0.
\end{cases}
\end{align}
\end{lemma}

\begin{proof}
The first two identities follow from \cite[Lemma 5.1]{SDK}.
We note that for any $d \in \mathbb{N}$ and $\text{Re} (s)>1 $,
\begin{align}\label{eq:domain_of_conv_1}
    \left| \frac{\log |\gcd(c,d)|}{c^s} \right| \le \left| \frac{\log |c|}{c^s} \right| = o(|c|^{-1+\varepsilon}), \quad c \to \infty 
\end{align}
for an arbitrary $\varepsilon>0$, and 
\[
|(\gcd(c,d))^{k}| \le 
\begin{cases}
1, \quad & \text{Re}(k) < 0, \\
|c^k|, \quad & \text{Re}(k) \ge 0,
\end{cases}
\]
thus 
\begin{align}\label{eq:domain_of_conv_2}
\left| \frac{(\gcd(c,d))^{k}}{  c^{s+k}   }  \right| \le \max \{ |c^{-s}|, |c^{-k-s}| \}.
\end{align}
The estimates \eqref{eq:domain_of_conv_1} and \eqref{eq:domain_of_conv_2} imply that for $\text{Re} (s)>1 $, $\text{Re} (s+k)>1$, the sums in \eqref{dirichlet_series_with_loggcd} and \eqref{dirichlet_series_with_gcd} converge. 
The right sides of both can be meromorphically continued by taking the meromorphic continuation of the Riemann zeta function and the meromorphic continuations of $\sum_{\ell | d} \Lambda(\ell) \ell^{-s}$ and $ \prod_{p | d} \left( 1 + (1-p^{-k})(p^{\oldksecond}+\ldots + p^{-v_{p}(d)s }) \right)$. The formula \eqref{eq:evaluations_gcd_dirichlet} follows from \eqref{dirichlet_series_with_gcd} by a direct substitution.
\end{proof}
We additionally need the following lemma that evaluates derivatives of meromorphic continuation of \eqref{dirichlet_series_with_gcd} at specific points.
\begin{corollary}\label{ref:lemma_about_Q=0} For $d\in\mathbb{N}$, 
    \begin{align}\label{eq:derivative1ofdanylodirichlet}
        \left. \frac{\partial}{\partial t_1} \left( \sum_{c \in \mathbb{N}} c^{t_1} \gcd(c,d)^{t_2} \right)\right|_{\substack{t_1 = \oldkfirst\\ t_2=0}} = - \zeta'(-\oldkfirst)
    \end{align}
    and 
        \begin{align}\label{eq:derivative2ofdanylodirichlet}        \left. \frac{\partial}{\partial t_2} \left( \sum_{c \in \mathbb{N}} c^{t_1} \gcd(c,d)^{t_2} \right)\right|_{\substack{t_1 = \oldkfirst\\ t_2=0}} = \zeta(- \oldkfirst) \sum_{\ell \divides d} \Lambda(\ell) \ell^{\oldkfirst}.
    \end{align}
\end{corollary}

\begin{proof}
We use \eqref{dirichlet_series_with_gcd} with $t_1=-s-k$ and $t_2=k$,
\begin{align}\label{eq:danylos_series_in_different_variables}
\sum_{c \in \mathbb{N}}  c^{t_1}  \gcd(c,d)^{t_2} = \zeta(-t_1) \prod_{p | d} \left( 1 + (1-p^{-t_2})(p^{t_1+t_2}+\ldots + p^{v_{p}(d) (t_1+t_2) }) \right).
\end{align}
To establish  \eqref{eq:derivative1ofdanylodirichlet}, we note 
\[
\frac{\partial}{\partial t_1} \left( 1 + (1-p^{-t_2})(p^{t_1+t_2}+\ldots + p^{v_{p}(d) (t_1+t_2) }) \right)\big|_{t_2=0} = 0.
\]
Thus, taking the derivative of both sides of  \eqref{eq:danylos_series_in_different_variables}, considering the meromorphic continuation, and substituting $t_2=0, t_1=\oldkfirst$, we obtain
\begin{align*}
  \frac{\partial}{\partial t_1}& \left(\left. \sum_{c \in \mathbb{N}} c^{t_1} \gcd(c,d)^{t_2} \right)\right|_{\substack{t_1 = \oldkfirst\\ t_2=0}}\\  &= - \left. \zeta'(-\oldkfirst) \prod_{p | d} \left( 1 + (1-p^{-t_2})(p^{t_1+t_2}+\ldots + p^{v_{p}(d) (t_1+t_2) }) \right)\right|_{\substack{t_1 = \oldkfirst\\ t_2=0}}\\
&= -\zeta'(-\oldkfirst),
\end{align*}
which implies \eqref{eq:derivative1ofdanylodirichlet}. In order to show  \eqref{eq:derivative2ofdanylodirichlet}, we note that \[
\left. \frac{\partial }{\partial t_2} (1-p^{-t_2}) \right|_{t_2 = 0} = p^{-t_2} \log p \big|_{t_2=0} = \log p
\]
and
\[
\left. 1+ (1-p^{-t_2}) (p^{t_1+t_2}+\ldots + p^{v_{p}(d) (t_1+t_2) }) \right|_{t_2=0} = 1.
\]
We thus get  
\begin{align*}
     \frac{\partial}{\partial t_2}  ( 1+ (1-p^{-t_2})& (p^{t_1+t_2}+\ldots + p^{v_{p}(d) (t_1+t_2) }) ) \big|_{t_2=0} \\
    &= \left. \frac{\partial }{\partial t_2} (1-p^{-t_2}) \right|_{t_2 = 0} \left. \cdot (p^{t_1+t_2}+\ldots + p^{v_{p}(d) (t_1+t_2) }) \right|_{t_2=0} \\
    & = \log (p) \cdot (p^{t_1} + \ldots + p^{v_p(d) t_1}).
\end{align*}
Summing over all possible prime divisors of $d$, we obtain 
\begin{align*}
    \zeta(-\oldkfirst) \frac{\partial}{\partial t_2}  & \left.  \left(  \prod_{p | d} (1+(1-p^{-t_2}) (p^{t_1+t_2}+\ldots + p^{v_{p}(d) (t_1+t_2) }) )\right)\right|_{\substack{t_1 = \oldkfirst\\ t_2=0}} \\
= & \zeta(-\oldkfirst) \sum_{p \divides d} \log (p) \cdot (p^{\oldkfirst} + \ldots + p^{v_p(d) \oldkfirst}).
\end{align*}
We can instead write this as a sum over all divisors of $d$ with the help of the von Mangoldt function
\[
\zeta(- \oldkfirst) \sum_{\ell \divides d} \Lambda(\ell) \ell^{\oldkfirst}
\]
yielding \eqref{eq:derivative2ofdanylodirichlet}.
\end{proof}

The following lemma allows us to rewrite the sums over $a, b \in \mathbb{Z} \setminus \{0\}$ with $ad-bc=n$ for a fixed $n \in \Z \setminus \{ 0\}$ as a sum over integers for  $c, d \in \mathbb{N}$. 
Additionally, we denote 
\begin{align}
v_{c,d} &:= \frac{c d}{\gcd(c,d)} \label{def:yn}
\end{align} and 
\[
\delta_{c | n} := \begin{cases}
1, \quad & \text{$c$ divides $n$,}\\
0, \quad & \text{$c$ does not divide $n$},
\end{cases}
\]
for  $c, d \in \mathbb{N}$.
We use the notation that $2$ divides $0$, thus we write  
$
\delta_{2 | 0} = 1.
$

\begin{lemma}\label{lemma:rewriting_stuff_in_terms_of_Hurwitz_zeta}	
  Let $n \in \Z \setminus \{ 0\}$, $c,d\in \mathbb{N}$,  $f: \mathbb{Z} \to \mathbb{C}$ and let  \begin{align*}
\mathcal{T}_n(c,d) & :=\sum_{\stackrel{a, b \in \mathbb{Z} \setminus \{0\}  }{ad-bc=n}}  f(bc),  
\end{align*} 
where the sum above is over all possible $a,b \in \mathbb{Z} \setminus \{0\}$ with $ad-bc = n$; if there exists no such $c,d$, then the sum is equal to zero.  
Then $\mathcal{T}_n(c,d)$ can be rewritten as 
\begin{equation}\label{eq:tncd}
\mathcal{T}_n(c,d) = \begin{cases}  \sum_{m \in \mathbb{Z}} f((m+u_{c,d}) v_{c,d}) - \delta_{c \divides n}f(-n)-\delta_{d \divides n} f(0), \quad & \gcd(c,d)  \divides n, \\
0, \quad & \gcd(c,d) \doesnotdivide n
\end{cases}
\end{equation}
for $u_{c,d}, v_{c,d} \in \mathbb{C}$ defined in \eqref{def:xn} and \eqref{def:yn}, respectively.
\end{lemma} 

\begin{proof}	
 If $\gcd(c,d) \doesnotdivide n$, then there do not exist $a, b \in \mathbb{Z}$ such that ${ad-bc=n}$ holds, and  $\mathcal{T}_n(c,d)=0$. We let $B, b^*$ and $a^*$ be as in \eqref{eq:definition_B} and \eqref{def:bstar}.
 
 We note that the set $B$  is parameterized by $m\in\mathbb{Z}$, and each of its elements takes the form
\begin{align}
b(m)=b^*+\frac{m}{\gcd(c,d)}d.
\end{align}
The corresponding $a(m) \in \mathbb{Z}$, defined by the property $a(m) d - b(m) c = n$, is equal to 
\begin{equation}\label{eq:amdef}
    a(m)
    =a^*+\frac{m}{\gcd(c,d)}c.
\end{equation}
We rewrite
\begin{align*}
b(m)c   = \left(\frac{ b^* \gcd(c,d)}{d} + m \right)\frac{ c d}{\gcd(c,d)} = (m+u_{c,d}) v_{c,d}
\end{align*}
and obtain
\[
\mathcal{T}_n(c,d) = 
\begin{cases}
  \sum^* f((m+u_{c,d}) v_{c,d}), \quad & \gcd(c,d) \divides n, \\ 
0, \quad & \gcd(c,d) \doesnotdivide  n,
\end{cases}
\]
where the sum $\sum^* $ is taken over all possible $m \in \mathbb{Z}$ such that $a(m) b(m) \neq 0$. We consider the following possibilities:
\begin{enumerate}[(i)]
	\item Let $d\divides n$, then the definition of $B$ implies  $ b^* = 0$. Thus the definition of $u_{c,d}$, \eqref{def:xn}, implies  $u_{c,d}=0$ and \[
	a(m)b(m) = \frac{m d}{\gcd(c,d)} \frac{n + \frac{m}{\gcd(c,d)} c d}{d} = \frac{m d}{\gcd(c,d)} \left(\frac{n}{d} + \frac{m c}{\gcd(c,d)}  \right).
	\]
	\begin{enumerate}[(1)]
	    \item If $c \doesnotdivide n$, then $\frac{n}{d} + \frac{m c}{\gcd(c,d)} =0$ has no integer solutions in $m$. Thus, the only integer $m$ such that $a(m)b(m)=0$ is $m = 0$. In this case \[
	f((m+u_{c,d}) v_{c,d})|_{m = 0} = f(0).
	\]
	    \item If $c \divides n$, then either $m=0$ or $m= - \frac{n \gcd(c,d)}{cd}$ in which case 
	    \[
	f((m+u_{c,d}) v_{c,d})|_{m  = - \frac{n \gcd(c,d)}{cd}} = f(-n).
	\]
	\end{enumerate}
	\item Let $d \doesnotdivide n$.
	\begin{enumerate}[(1)]
	\item If  $c \doesnotdivide n$, then $a(m) b(m)$ cannot be equal to zero.
	\item If $c \divides  n$, then $b^* = - \frac{n}{c}$, and $a^*=0$. Moreover, 	
	  when 
	$a(m) b(m) = 0 $ then $a(m) = 0$ and so $m = 0.$
	\end{enumerate}
\end{enumerate}
Thus, if $d\divides n$, the term $f(0)$ has to be omitted, and if $c\divides n$, the term $f(-n)$ has to be omitted. That implies \eqref{eq:tncd} and  finishes the proof. 
\end{proof}

	\section{Proofs of main theorems}\label{sec:some_precise_results}

The main goals of this section is to rigorously evaluate \eqref{eq:sums} and prove Theorems \ref{thm:final_exact_result} and \ref{thm:final_exact_result2}. 

\subsection{Proof of Theorem \ref{thm:final_exact_result}} 

We  first 
need the following lemma:

	\begin{lemma}\label{lemma:ldn_is_equal_to_something_poly}
		Let $n \in \mathbb{Z} \setminus \{0\}$, $\oldksecond \in \R$,  $d \in \mathbb{N}$ and $\powerpoly \in \mathbb{Z}$ with $\powerpoly \ll_{\oldksecond} 0$, then the sum
		\begin{align}\label{eq:sum_over_c_of_bc_polynomial}
		\sum_{c \in \mathbb{N}} c^{\oldksecond} \sum_{\stackrel{a, b \in \mathbb{Z} \setminus \{0\}  }{ad-bc=n}} (bc)^\powerpoly
		\end{align}
	 equals 
		\[
		(-1)^{\powerpoly+1} n^\powerpoly \sigma_{\oldksecond}(n) + 2 d^\powerpoly \delta_{2 \divides \powerpoly} \zeta(-\powerpoly) \zeta(-\oldksecond-\powerpoly) \prod_{p | d} \left( 1 + (1-p^{\powerpoly})(p^{\oldksecond}+\ldots + p^{v_{p}(d) \oldksecond }) \right),
		\]
	for $d \divides n$ 	and	 equals 
		\begin{align*} (-1)^{\powerpoly+1} n^\powerpoly \sigma_{\oldksecond}(n)
	+  d^{\oldksecond+2\powerpoly}   \sum_{\stackrel{0<c'\le d}{\gcd(c',d)|n}}  \frac{ \zeta(-\oldksecond-\powerpoly,\frac{c'}{d})}{\gcd(c',d)^{\powerpoly}}  \sum_{m\in \mathbb{Z}} (m+u_{c',d})^{\powerpoly},
	\end{align*}
	for $d \doesnotdivide n$.
	\end{lemma}
	\begin{proof}	[Proof of Lemma \ref{lemma:ldn_is_equal_to_something_poly}]	
		In what follows, we consider two different cases: $d \divides n$ and $d \doesnotdivide  n$.
		
		First  assume $d\divides n$. 
		In this case, we have $\gcd(c,d) \divides n$, and the definition of $b^*$ in  \eqref{def:bstar} implies  $b^*=0$. Hence, by \eqref{def:yn}, $u_{c,d} = 0$. 
		With the help of Lemma \ref{lemma:rewriting_stuff_in_terms_of_Hurwitz_zeta}, we write 
		\begin{align}
		\sum_{\stackrel{a, b \in \mathbb{Z} \setminus \{0\}  }{ad-bc=n}} (bc)^\powerpoly &
		=\sum_{m \in \mathbb{Z} \setminus \{0\} }(m v_{c,d})^\powerpoly- (-n)^\powerpoly \delta_{c \divides n} \nonumber
	\\ & 
		\stackrel{\eqref{def:yn}}{=}2 \delta_{2 \divides \powerpoly} \left( \frac{c d}{\gcd(c,d)} \right)^{\powerpoly} \zeta(-\powerpoly) + (-1)^{\powerpoly+1} \delta_{c \divides n} n^\powerpoly.  \label{eq:intermediate_sum_of_polynomials_overad-bc}
		\end{align}
		Multiplying by $c^{\oldksecond}$ and summing over $c \in \mathbb{N}$, we obtain 
		\begin{align}
		 \sum_{c \in \mathbb{N}}c^{\oldksecond} &\left( 2 \delta_{2 \divides \powerpoly} \left( \frac{ c d}{\gcd(c,d)} \right)^{\powerpoly} \zeta(-\powerpoly) + (-1)^{\powerpoly+1} \delta_{c \divides n} n^\powerpoly \right) \nonumber \\
		&= 2 d^{\powerpoly}  \delta_{2 \divides \powerpoly} \zeta(-\powerpoly)\sum_{c \in \mathbb{N}}  \frac{c^{\powerpoly +\oldksecond}}{\gcd(c,d)^{\powerpoly}}   + (-1)^{\powerpoly+1} n^\powerpoly \sigma_{\oldksecond}(n). \label{eq:intermediate_sum_of_polynomials_overad-bc2}
		\end{align}
	Finally, Lemma  \ref{lemma:Lemma5.1inKLSR} gives
		\[\sum_{c \in \mathbb{N}} \frac{c^{\powerpoly +\oldksecond}}{\gcd(c,d)^{\powerpoly}} =
		\zeta(-\oldksecond-\powerpoly) \prod_{p | d} \left( 1 + (1-p^{\powerpoly})(p^{\oldksecond}+\ldots + p^{v_{p}(d) \oldksecond }) \right),\]
		and, together with \eqref{eq:intermediate_sum_of_polynomials_overad-bc2}, this implies the first statement of the lemma.
		
		When $d \doesnotdivide n$ and $\gcd(c,d) \divides n$, \begin{align}
		\sum_{\stackrel{a, b \in \mathbb{Z} \setminus \{0\}  }{ad-bc=n}} (bc)^\powerpoly &=  (-1)^{\powerpoly+1} \delta_{c \divides n} n^\powerpoly + \sum_{m \in \mathbb{Z}} ((m+u_{c,d}) v_{c,d})^\powerpoly  \nonumber \\
		& \stackrel{\eqref{def:yn}}{=}  (-1)^{\powerpoly+1} \delta_{c \divides n} n^\powerpoly + \left(\frac{ c d}{\gcd(c,d)}\right)^{\powerpoly} \sum_{m \in \mathbb{Z}} (m+u_{c,d})^{\powerpoly}   ,
		 \label{eq:d_not_dividing_n_intermediate_equality}
		\end{align}
		and when $d \doesnotdivide n$ and $\gcd(c,d) \doesnotdivide n$,
		\[
		\sum_{\stackrel{a, b \in \mathbb{Z} \setminus \{0\}  }{ad-bc=n}} (bc)^\powerpoly = 0.
		\] 
		When $d\doesnotdivide n$, multiplying \eqref{eq:d_not_dividing_n_intermediate_equality} by $c^{\oldksecond}$ and summing over $c$, we  obtain
		\begin{align}\label{eq:summingoverc}
	\sum_{c \in \mathbb{N}} c^{\oldksecond} \sum_{\stackrel{a, b \in \mathbb{Z} \setminus \{0\}  }{ad-bc=n}} (bc)^\powerpoly	& 
		=(-1)^{\powerpoly+1} n^\powerpoly \sigma_{\oldksecond}(n)+ \sum_{\stackrel{c \in \mathbb{N}}{\gcd(c,d) \divides n}} \frac{c^{\oldksecond+\powerpoly} d^\powerpoly}{\gcd(c,d)^\powerpoly}   \sum_{m \in \mathbb{Z}}  (m+u_{c,d})^{\powerpoly} .
		\end{align}
We note that for any function $f: \mathbb{N} \to \mathbb{C}$,
		\begin{equation}\label{eq:regroupping}
		\sum_{\stackrel{c \in \mathbb{N}}{\gcd(c,d) \divides n}} f(c) =  \sum_{ \stackrel{0 < c' \le d}{ \gcd(c', d) \divides n} }\sum_{j=0}^\infty f(jd +c'),
		\end{equation}
		and for any $j \in \mathbb{N}_0$  and $c' \in \mathbb{Z}$,
		\begin{equation}\label{eq:gcd_doesnt_change_much}
		\gcd(jd+c',d) = \gcd(c',d) \quad\text {and} \quad u_{c',d}=u_{jd+c',d}.
		\end{equation}
		Rewriting the second term in \eqref{eq:summingoverc} with the help of \eqref{eq:regroupping},  we get 
		\begin{align*}
		 \sum_{\stackrel{0<c'\le d}{\gcd(c'+jd,d)|n}} \sum_{j=0}^\infty & \frac{(jd+c')^{\oldksecond+\powerpoly} d^{\powerpoly}    }{\gcd(c'+jd,d)^{\powerpoly}}   \sum_{m\in \mathbb{Z}} (m+u_{c'+jd,d})^{\powerpoly} \\
		 & \stackrel{\eqref{eq:gcd_doesnt_change_much}}{=}
		\sum_{\stackrel{0<c'\le d}{\gcd(c',d)|n}} \sum_{j=0}^\infty  \frac{ (jd+c')^{\oldksecond+\powerpoly} d^{\powerpoly}}{\gcd(c',d)^{\powerpoly}} \sum_{m\in \mathbb{Z}}   (m+u_{c',d})^{\powerpoly}  \\
		&=  d^{\oldksecond+2\powerpoly}   \sum_{\stackrel{0<c'\le d}{\gcd(c',d)|n}}  \frac{ \zeta(-\oldksecond-\powerpoly,\frac{c'}{d})}{\gcd(c',d)^{\powerpoly}}   \sum_{m\in \mathbb{Z}} (m+u_{c',d})^{\powerpoly},		\end{align*}
		where $u_{c,d}$ is defined as in \eqref{def:xn}. 
		This implies the second statement of the Lemma.
	\end{proof}

	\begin{proof}[Proof of Theorem \ref{thm:final_exact_result}]
Assuming the factorization 
	\[
	n_1 = a d, \quad n_2 =  -b c, \quad a, b, c, d \in \mathbb{Z},
	\]
	we rewrite each summand in the left hand side  of \eqref{eq:target_sum1} as 
	\begin{align*}
	\sum_{d \in \mathbb{N}} d^{\oldkfirst} \sum_{c \in \mathbb{N} } c^{\oldksecond}  \sum_{\stackrel{a,b\in \mathbb{Z} \setminus \{0\}}{ad - bc =n}} (-bc)^\powerpoly.
	\end{align*}
Then we  multiply the formula in  Lemma \ref{lemma:ldn_is_equal_to_something_poly} by $(-1)^\powerpoly$ and obtain \eqref{eq:target_sum1}.
	
	Now, we need to deal with the meromorphic continuation. We note that  $2 \delta_{2 \divides \powerpoly}$ can be meromorphically continued by taking $1+e^{2 \pi i \powerpoly}$. 
	We consider the second term on the right side of~\eqref{eq:target_sum1}. There are three possible cases:
	\begin{enumerate}
		\item If $\oldksecond+\powerpoly \in 2\mathbb{N}+1$, then $\delta_{2 \divides n}$ vanishes.
		\item If $\oldksecond+\powerpoly \in 2\mathbb{N}$, then $\zeta(-\oldksecond-\powerpoly)$ vanishes.
		\item If $\oldksecond+\powerpoly=0$, then $\oldksecond \in 2 \mathbb{Z}$ implies $\powerpoly \in 2 \mathbb{Z}$. Using \eqref{eq:evaluations_gcd_dirichlet}, we obtain that the second line in the right hand side of \eqref{eq:target_sum1} becomes $(-1)^{\powerpoly+1} \zeta(\oldksecond) \sigma_{\oldkfirst}(n)$. Since we assumed $\powerpoly = - \oldksecond$ and $\oldksecond$ is even, $(-1)^{\powerpoly+1}=-1$.
	\end{enumerate}
	
 It remains to show, for $d \doesnotdivide n$,   
	\begin{align}\label{imposter_that_pretends_to_vanish_but_wil_probably_become_a_factor_of_L_functions}
	\sum_{\stackrel{0<c'\le d}{\gcd(c',d)|n}}  \frac{ \zeta(-\oldksecond-\powerpoly,\frac{c'}{d})}{\gcd(c',d)^{\powerpoly}}  \sum_{m\in \mathbb{Z}} (m+u_{c',d})^{\powerpoly} = 0.
	\end{align}
	 We note the element $c'=d$ is not present in the sum \eqref{imposter_that_pretends_to_vanish_but_wil_probably_become_a_factor_of_L_functions} because, in this case, $\gcd(c',d)=d$ divides $n=ad-bc$, but this contradicts the assumption $d \doesnotdivide  n$.	
	For any other $c'$, there is a pair $d-c'$ in the sum $0 < c' \le d$ (for $c'=d/2$, we consider~$c'$ to be its own pair). We note that  
$	\gcd(d-c',d)=\gcd(c',d)$	and 
$	u_{d-c',d} = -u_{c',d}  
$	 imply
	\begin{align}
	 \frac{ \zeta(-\oldksecond-\powerpoly,1-\frac{c'}{d})}{\gcd(d-c',d)^{\powerpoly}}  \sum_{m\in \mathbb{Z}} (m+u_{d-c',d})^{\powerpoly} & = 	 \frac{ \zeta(-\oldksecond-\powerpoly,1-\frac{c'}{d})}{\gcd(c',d)^{\powerpoly}}  \sum_{m\in \mathbb{Z}} (m-u_{c',d})^{\powerpoly} \nonumber \\
	 & \stackrel{\eqref{eq:lemma2.2full_sum}}{=} 	(-1)^\powerpoly  \frac{ \zeta(-\oldksecond-\powerpoly,1-\frac{c'}{d})}{\gcd(c',d)^{\powerpoly}}   \sum_{m\in \mathbb{Z}} (m+u_{c',d})^{\powerpoly}. \label{eq:parity_stuff}
	\end{align}
We can rewrite  \eqref{imposter_that_pretends_to_vanish_but_wil_probably_become_a_factor_of_L_functions} by grouping together contributions from   the elements $c'$ and $d-c'$ and applying \eqref{eq:parity_stuff} to get
	\[
	\frac{1}{2}\sum_{\stackrel{0<c'\le d}{\gcd(c',d)|n}}  \frac{ \zeta(-\oldksecond-\powerpoly,\frac{c'}{d})+(-1)^\powerpoly \zeta(-\oldksecond-\powerpoly,1-\frac{c'}{d})}{\gcd(c',d)^{\powerpoly}}  \cdot \sum_{m\in \mathbb{Z}} (m+u_{c',d})^{\powerpoly}.
	\]
	In turn, \eqref{ref:vanishing_for_bernoulli} implies that for $\oldksecond+\powerpoly \in \mathbb{N}_0$ and $\oldksecond \in 2 \mathbb{Z}$, the sum above vanishes. 
\end{proof}

\subsection{Proof of Theorem \ref{thm:final_exact_result2}}

In order to prove Theorem \ref{thm:final_exact_result2}, we need the following lemma.

	\begin{lemma}\label{lemma:ldn_is_equal_to_something_log}
		Let $n\in\N$, $d \in \mathbb{N}$, $\powerlog \in \mathbb{Z}$ with $\powerlog \ll_{r_2} 0$, and $\text{Re}(t_2)<-2-Q$, then the sum
		\begin{align}\label{eq:sum_over_c_of_bc_logarithm}
		\sum_{c \in \mathbb{N}} c^{\oldksecond} \sum_{\stackrel{a, b \in \mathbb{Z} \setminus \{0\}  }{ad-bc=n}} (bc)^\powerlog \log |bc|
		\end{align}
	equals 
		\begin{align}\label{eq:sums_with_log_d_divides_n_part}
		&    (-1)^{\powerlog+1} \sigma_{\oldksecond}(n) n^\powerlog \log |n|  \\
		& \ \ \ \ \  + 2 \delta_{2 \divides \powerlog} d^\powerlog \zeta(-\powerlog) \left( \frac{\partial}{\partial t_1} + \frac{\partial}{\partial t_2} \right) \left. \left( \sum_{c \in \N} c^{t_1} \gcd(c,d)^{t_2} \right)\right|_{\substack{t_1=\oldksecond+\powerlog\\ t_2=-\powerlog}} \nonumber \\
		& \ \ \ \ \ + 2 \delta_{2 \divides \powerlog} d^\powerlog (\log(d) \zeta(-\powerlog) -\zeta'(-\powerlog))  \sum_{c \in \N} \frac{c^{\oldksecond+\powerlog}}{\gcd(c,d)^\powerlog}, \nonumber 
		\end{align}
		for $d \divides n$, and	 equals 
		\begin{align*} 
		&  (-1)^{\powerlog+1} \sigma_{\oldksecond}(n) n^\powerlog \log |n| \nonumber \\
		&\ \ \ \ \  +  \left(\frac{ c d}{\gcd(c,d)}\right)^{\powerlog} \log \left| \frac{ c d}{\gcd(c,d)} \right| \sum_{m \in \mathbb{Z}} (m+u_{c,d})^{\powerlog} \nonumber \\
		&\ \ \ \ \  + \frac{d^{\oldksecond+ 2 \powerlog}}{(\gcd(c,d))^\powerlog} \sum_{\stackrel{0 < c' \le d}{\gcd(c'+jd, d) \divides n}} \zeta(-\oldksecond-\powerlog, \tfrac{c'}{d})  \sum_{m \in \mathbb{Z}} \frac{ \log|m + u_{c'+jd,d}|}{(m+u_{c'+jd,d})^{-\powerlog}}
\end{align*}
for $d \doesnotdivide n$.
\end{lemma}

\begin{proof}[Proof of Lemma \ref{lemma:ldn_is_equal_to_something_log}]
As in the proof of Lemma \ref{lemma:ldn_is_equal_to_something_poly}, we use Lemma \ref{lemma:rewriting_stuff_in_terms_of_Hurwitz_zeta} and note again that when $d\divides n$, we have  $b^*=0$. Thus  for $d \divides n$, we write 
		\begin{align}
		\sum_{\stackrel{a, b \in \mathbb{Z} \setminus \{0\}  }{ad-bc=n}} (bc)^\powerlog \log|bc| 
		= & \sum_{m \in \mathbb{Z} \setminus \{0\} }(m v_{c,d})^\powerlog \log|m v_{c,d}| - (-n)^\powerlog \log|n| \delta_{c \divides n} \nonumber
	\\  
		\stackrel{\eqref{def:yn}}{=} & (-1)^{\powerlog+1} \delta_{c \divides n} n^\powerlog \log |n| \nonumber \\ 
		&\ + 2 \delta_{2 \divides \powerlog} \left( \frac{c d}{\gcd(c,d)} \right)^{\powerlog} \left(\log\left(\frac{c d}{\gcd(c,d)} \right) \zeta(-\powerlog)-\zeta'(-\powerlog) \right) \nonumber \\
	=	& (-1)^{\powerlog+1} \delta_{c \divides n} n^\powerlog \log |n| \nonumber \nonumber \\
		&\ + 2 \delta_{2 \divides \powerlog} \frac{(cd)^\powerlog}{\gcd(c,d)^\powerlog}\log(c) \zeta(-\powerlog) \nonumber \\
		&\ - 2 \delta_{2 \divides \powerlog} \frac{(cd)^\powerlog}{\gcd(c,d)^\powerlog}\log(\gcd(c,d)) \zeta(-\powerlog) \nonumber \\
		& \ + 2 \delta_{2 \divides \powerlog} \frac{(cd)^\powerlog}{\gcd(c,d)^\powerlog}(\log(d) \zeta(-\powerlog) -\zeta'(-\powerlog))\nonumber.
		\end{align}

We multiply the expression above by $c^{\oldksecond}$ and sum over $c \in \mathbb{N}$. The first term becomes 
\begin{align*}
\sum_{c \in \mathbb{N}} c^{\oldksecond} (-1)^{\powerlog+1} \delta_{c \divides n} n^\powerlog \log |n| &= (-1)^{\powerlog+1}  n^\powerlog \log |n| \sum_{c \in \N} c^{\oldksecond} \delta_{c \divides n} \\
&= (-1)^{\powerlog+1}  n^\powerlog \log |n| \sigma_{\oldksecond}(n).
\end{align*}

The second line becomes 
\begin{align}
    \sum_{c \in \mathbb{N}} c^{\oldksecond}  2 \delta_{2 \divides \powerlog} \frac{(cd)^\powerlog}{\gcd(c,d)^\powerlog}\log(c) \zeta(-\powerlog) \nonumber 
    &= 2 \delta_{2 \divides \powerlog} d^\powerlog \zeta(-\powerlog) \sum_{c \in \N} \frac{c^{\oldksecond+\powerlog}}{\gcd(c,d)^\powerlog} \log(c)\nonumber\\
    &= 2 \delta_{2 \divides \powerlog} d^\powerlog \zeta(-\powerlog) \frac{\partial}{\partial t_1} \left. \left( \sum_{c \in \N} c^{t_1} \gcd(c,d)^{t_2} \right)\right|_{\substack{t_1=\oldksecond+\powerlog\\t_2=-\powerlog}} \label{eq:summand_in_blawdkjfhk}
\end{align}
since $\displaystyle\sum_{c \in \N} \frac{c^{\oldksecond+\powerlog}}{\gcd(c,d)^\powerlog} \log(c)$ is absolutely convergent for $\text{Re}(t_2)<-2-Q$.

The third line becomes 
\begin{align}
     - \sum_{c \in \N} c^{\oldksecond}  2 \delta_{2 \divides \powerlog}& \frac{(cd)^\powerlog}{\gcd(c,d)^\powerlog}\log(\gcd(c,d)) \zeta(-\powerlog) \nonumber \\
    &= - 2 \zeta(-\powerlog)  \delta_{2 \divides \powerlog} d^\powerlog \sum_{c \in \N}    \frac{ c^{\oldksecond+\powerlog} }{\gcd(c,d)^\powerlog}\log(\gcd(c,d)) \label{eq:summandfwsdwr}\\
    &= 2 \zeta(-\powerlog)  \delta_{2 \divides \powerlog} d^\powerlog \frac{\partial}{\partial t_2} \left(  \left.  \sum_{c \in \N} c^{t_1} \gcd(c,d)^{t_2} \right) \right|_{\substack{t_1=\oldksecond+\powerlog\\t_2=-\powerlog}} \nonumber 
\end{align}
since $\displaystyle\sum_{c \in \N} \frac{c^{\oldksecond+\powerlog}}{\gcd(c,d)^\powerlog} \log(c)$ is absolutely convergent for $\text{Re}(t_2)<-2-Q$.

The fourth line becomes 
\begin{align}
    \sum_{c \in \N} c^{\oldksecond} 2 \delta_{2 \divides \powerlog}& \frac{(cd)^\powerlog}{\gcd(c,d)^\powerlog}(\log(d) \zeta(-\powerlog) -\zeta'(-\powerlog)) \nonumber \\
    &= 2 \delta_{2 \divides \powerlog} d^\powerlog (\log(d) \zeta(-\powerlog) -\zeta'(-\powerlog))  \sum_{c \in \N} \frac{c^{\oldksecond+\powerlog}}{\gcd(c,d)^\powerlog}. \label{eq:summand_aldfkjlskfg342}
\end{align}

When $d \doesnotdivide n$ and $\gcd(c,d) \divides n$, \begin{align}
		\sum_{\stackrel{a, b \in \mathbb{Z} \setminus \{0\}  }{ad-bc=n}} (bc)^\powerlog \log|bc| &=  (-1)^{\powerlog+1} \delta_{c \divides n} n^\powerlog \log|n| + \sum_{m \in \mathbb{Z}} ((m+u_{c,d}) v_{c,d})^\powerlog \log| (m+u_{c,d}) v_{c,d}|  \nonumber \\
		& \stackrel{\eqref{def:yn}}{=} (-1)^{\powerlog+1} \delta_{c \divides n} n^\powerlog \log|n| \nonumber \\
		&\ \ \ \ \ +  \left(\frac{ c d}{\gcd(c,d)}\right)^{\powerlog} \log \left| \frac{ c d}{\gcd(c,d)} \right| \sum_{m \in \mathbb{Z}} (m+u_{c,d})^{\powerlog} \nonumber \\
		& \ \ \ \ \ + \left(\frac{ c d}{\gcd(c,d)}\right)^{\powerlog} \sum_{m \in \mathbb{Z}} (m+u_{c,d})^{\powerlog} \log|m + u_{c,d}|.   \label{eq:d_not_dividing_n_intermediate_equality_for_log}
		\end{align}
Multiplying by $c^{\oldksecond}$ and summing over $c$, we leave the first and the second line as it is and rewrite the third line similar to the proof of  Lemma \ref{lemma:ldn_is_equal_to_something_poly},
\begin{align*}
\sum_{c \in \N} c^{\oldksecond} & \left(\frac{ c d}{\gcd(c,d)}\right)^{\powerlog} \sum_{m \in \mathbb{Z}} (m+u_{c,d})^{\powerlog} \log|m + u_{c,d}| \\
&= \sum_{j=0}^\infty \sum_{\stackrel{0 < c' \le d}{\gcd(c'+jd, d) \divides n}}  \frac{(jd+c')^{\oldksecond+\powerlog} d^\powerlog}{(\gcd(c,d))^\powerlog} \sum_{m \in \mathbb{Z}} (m+u_{c'+jd,d})^\powerlog \log|m + u_{c'+jd,d}| \\
&=  \frac{d^{\oldksecond+ 2 \powerlog}}{(\gcd(c,d))^\powerlog} \sum_{\stackrel{0 < c' \le d}{\gcd(c', d) \divides n}} \zeta(-\oldksecond-\powerlog, \tfrac{c'}{d})  \sum_{m \in \mathbb{Z}} (m+u_{c',d})^\powerlog \log|m + u_{c',d}|.
\end{align*}
\end{proof}

\begin{proof}[Proof of Theorem \ref{thm:final_exact_result2}]
    The theorem follows from Lemma \ref{lemma:ldn_is_equal_to_something_log} similarly to the proof of Theorem \ref{thm:final_exact_result}. 
    
  The first term in the right side of \eqref{eq:target_sum2} admits a meromorphic continuation trivially. 
For the second and the third summands in \eqref{eq:target_sum2}, we consider the cases $\powerlog \in \mathbb{N}$ and $\powerlog=0$ separately.
    
    For $\powerlog \in \N$, $\delta_{2 \divides \powerlog} \zeta(-\powerlog)$ vanishes. This implies that what remains from the second and the third terms is 
    \begin{align}\label{eq:blah}
    -    (-1)^\powerlog 2 \delta_{2 \divides \powerlog} \zeta'(-\powerlog) \sum_{d \divides \powerlog} \sum_{c \in \mathbb{N}} \frac{c^{\oldksecond+\powerlog}}{\gcd(c,d)^\powerlog}.
    \end{align}
   By Lemma \ref{lemma:Lemma5.1inKLSR}, \eqref{eq:blah} is a multiple of $\delta_{2 \divides \powerlog}  \zeta(-\oldksecond-\powerlog)$. If we additionally assume that  $\powerlog+\oldksecond \in \mathbb{N}$, then $\delta_{2 \divides \powerlog}  \zeta(-\oldksecond-\powerlog) = 0$, and \eqref{eq:blah} vanishes. 
    
    For $\powerlog = 0$, by Lemma \ref{ref:lemma_about_Q=0}, \eqref{eq:summand_in_blawdkjfhk} becomes 
    \[
-2 \zeta(0) \zeta'(-\oldksecond) = \zeta'(-\oldksecond).
    \]
    For $\powerlog = 0$, by Lemma \ref{ref:lemma_about_Q=0},  \eqref{eq:summandfwsdwr} becomes  
\[
-2 \zeta(0) \sum_{c \in \mathbb{N}} c^{\oldksecond} \log|\gcd(c,d)| = \sum_{c \in \mathbb{N}} c^{\oldksecond} \log|\gcd(c,d)| = \zeta(-\oldksecond) \sum_{\ell \divides d} \Lambda(\ell) \ell^{\oldksecond}.
\]
For $\powerlog = 0$, \eqref{eq:summand_aldfkjlskfg342} becomes 
\[
2 (\log(d) \zeta(0)-\zeta'(0)) \zeta(-\oldksecond) = - \zeta(-\oldksecond)  \log \left(\frac{d}{2 \pi }\right).
\]
Combining these three terms, we obtain
\[
\zeta'(-\oldksecond)+\zeta(-\oldksecond) \sum_{\ell \divides d} \Lambda(\ell) \ell^{\oldksecond} - \zeta(-\oldksecond) \log \left(\frac{d}{2 \pi }\right).
\]
For $\oldksecond>0$, this is equal to $\zeta'(-\oldksecond)$, and for $\oldksecond=0$, this is equal to $- \log(2 \pi)$.

Finally, the sum  
    \[
    \sum_{m \in \mathbb{Z}} (m+u_{c,d})^\powerlog
    \]
    vanishes by \eqref{eq:crazy_vanihsing}. The inner sum in the last term on the right side of \eqref{eq:target_sum2} vanishes by the similar argument as for  \eqref{imposter_that_pretends_to_vanish_but_wil_probably_become_a_factor_of_L_functions} in Theorem \ref{thm:final_exact_result}.
\end{proof}
\bibliography{draft3} 

\providecommand{\bysame}{\leavevmode\hbox to3em{\hrulefill}\thinspace}
\providecommand{\MR}{\relax\ifhmode\unskip\space\fi MR }
\providecommand{\MRhref}[2]{%
  \href{http://www.ams.org/mathscinet-getitem?mr=#1}{#2}
}
\providecommand{\href}[2]{#2}
\begin{thebibliography}{10}

\bibitem{Blomer2004}
V.~Blomer, \emph{{Shifted convolution sums and subconvexity bounds for
  automorphic L-functions}}, International Mathematics Research Notices
  \textbf{2004} (2004), no.~73, 3905--3926.

\bibitem{CGPWW2021}
S.~M. Chester, M.~B. Green, S.~S. Pufu, Y.~Wang, and C.~Wen, \emph{New modular
  invariants in {$\mathcal{N}=4$} super-{Y}ang-{M}ills theory}, J. High Energy
  Phys. (2021), no.~4, Paper No. 212, 56.

\bibitem{cho2013evaluation}
B.~Cho, D.~Kim, and H.~Park, \emph{Evaluation of a certain combinatorial
  convolution sum in higher level cases}, Journal of Mathematical Analysis and
  Applications \textbf{406} (2013), no.~1, 203--210.

\bibitem{MR4163822}
N.~Diamantis, \emph{Kernels of {$L$}-functions and shifted convolutions}, Proc.
  Amer. Math. Soc. \textbf{148} (2020), no.~12, 5059--5070. \MR{4163822}

\bibitem{NIST}
\emph{{\it NIST Digital Library of Mathematical Functions}},
  http://dlmf.nist.gov/, Release 1.1.6 of 2022-06-30, F.~W.~J. Olver, A.~B.
  {Olde Daalhuis}, D.~W. Lozier, B.~I. Schneider, R.~F. Boisvert, C.~W. Clark,
  B.~R. Miller, B.~V. Saunders, H.~S. Cohl, and M.~A. McClain, eds.

\bibitem{FKL}
K.~Fedosova and K.~Klinger-Logan, \emph{{Whittaker Fourier type solutions to
  differential equations arising from string theory}}.

\bibitem{GMV2015}
M.~B. Green, S.~D. Miller, and P.~Vanhove,
  \emph{{$SL(2,\mathbb{Z})$}-invariance and {D}-instanton contributions to the
  {$D^6R^4$} interaction}, Commun. Number Theory Phys. \textbf{9} (2015),
  no.~2, 307--344.

\bibitem{hahn2007convolution}
H.~Hahn, \emph{Convolution sums of some functions on divisors}, The Rocky
  Mountain Journal of Mathematics (2007), 1593--1622.

\bibitem{huard2002elementary}
J.~G. Huard, Z.~M. Ou, B.~K. Spearman, and K.~S. Williams, \emph{Elementary
  evaluation of certain convolution sums involving divisor functions}, Number
  theory for the Millennium \textbf{2} (2002), 229--274.

\bibitem{kim2013convolution}
A.~Kim, D.~Kim, and L.~Yan, \emph{Convolution sums arising from divisor
  functions}, Journal of the Korean Mathematical Society \textbf{50} (2013),
  no.~2, 331--360.

\bibitem{SDK}
K.~Klinger-Logan, S.~Miller, and D.~Radchenko, \emph{The {$D^6R^4$} interaction
  as {Poincar\'{e}} series, and a related shifted convolution sum}, preprint
  (2022).

\bibitem{lemire2006evaluation}
M.~Lemire and K.~S. Williams, \emph{Evaluation of two convolution sums
  involving the sum of divisors function}, Bulletin of the Australian
  Mathematical Society \textbf{73} (2006), no.~1, 107--115.

\bibitem{Michel2004}
P.~Michel, \emph{The subconvexity problem for {Rankin-Selberg L}-functions and
  equidistribution of {Heegner} points}, Annals of Mathematics \textbf{160}
  (2004), no.~1, 185--236.

\bibitem{ntienjem2017evaluation}
E.~Ntienjem, \emph{Evaluation of the convolution sum involving the sum of
  divisors function for 22, 44 and 52}, Open Mathematics \textbf{15} (2017),
  no.~1, 446--458.

\bibitem{park2022multinomial}
H.~Park, \emph{The multinomial convolution sum of a generalized divisor
  function}, Open Mathematics \textbf{20} (2022), no.~1, 419--430.

\bibitem{Ramanujan}
S.~Ramanujan, \, \emph{On certain arithmetical functions}, Trans. Cambridge
  Philos. Soc. (1916), no.~22, 159 -- 184.

\end{thebibliography}
\bibliographystyle{amsplain}

\end{document}